\documentclass{amsart}
\title[Non-invariant deformations of left-invariant complex structures]
{Non-invariant deformations of left-invariant complex structures on compact Lie groups
}

\author{Hiroaki Ishida}
\address{Department of Mathematics and Computer Science, Graduate School of Science and Engineering, Kagoshima University}
\email{ishida@sci.kagoshima-u.ac.jp}
\author{Hisashi Kasuya}
\address{Department of Mathematics, Graduate School of Science, Osaka University, Osaka, Japan.}
\email{kasuya@math.sci.osaka-u.ac.jp}

\thanks{The first author is supported by JSPS KAKENHI Grant Number JP20K03592. The second author is supported by JSPS KAKENHI Grant Number JP19H01787, JP21K03248}

\date{\today}

\usepackage{amssymb}
\usepackage{amsmath}
\usepackage{amscd}
\usepackage{amstext}
\usepackage{amsfonts}
\usepackage{color}
\usepackage[all]{xy}

\newtheorem{prop}{Proposition}[section]
\newtheorem*{prop*}{Proposition \referenza}
\newtheorem{thm}[prop]{Theorem}
\newtheorem*{thm*}{Theorem \referenza}
\newtheorem{cor}[prop]{Corollary}
\newtheorem*{cor*}{Corollary \referenza}

\theoremstyle{definition}
\newtheorem{defi}[prop]{Definition}
\newtheorem{rem}[prop]{Remark}
\newtheorem{ex}[prop]{Example}

\newcommand{\C}{\mathbb{C}}

\newcommand{\N}{\mathbb{N}}
\newcommand{\Q}{\mathbb{Q}}

\newcommand{\g}{\frak{g}}
\newcommand{\m}{\frak{m}}
\newcommand{\gl}{\frak{gl}}

\newcommand{\h}{\frak{h}}
\newcommand{\fk}{\mathfrak{k}}
\newcommand{\fb}{\frak{b}}
\newcommand{\ft}{\frak{t}}
\newcommand{\fl}{\frak{l}}
\newcommand{\fu}{\frak{u}}

\newcommand{\Id}{\operatorname{Id}}

\newcommand{\relmiddle}[1]{\mathrel{}\middle#1\mathrel{}}

\begin{document} 

\maketitle
\begin{abstract}
We give small deformations of a left-invariant complex structure on each simply connected  semisimple compact Lie group of even dimension which are not biholomorphic to any left-invariant (right-invariant) complex structure by using the Kuranishi space.
On such deformed complex manifolds,  we prove the Borel-Weil-Bott type theorem  and  we compute the cohomology of holomorphic tangent bundles.
\end{abstract}

\section{Introduction}
In \cite{B} (see also \cite[Section 4.8]{Ak}), combining the generalization of the Borel-Weil theorem and Kodaira-Spencer theory, Bott proves that the complex structure of every flag manifold is locally rigid.
Therefore, small deformations of the flag manifolds are uninteresting.
In contrast to the flag case, small deformations  of general compact complex homogeneous  manifolds may be non-trivial and  very interesting.
For examples, compact complex parallelizable manifolds (complex manifolds with trivial holomorphic tangent bundles)  are  complex homogeneous by Wang's result \cite{Wangp} (see  also \cite[Section 3.4]{Ak}) and 
many  complex parallelizable manifolds admit non-trivial deformations and there are various results on the geometry of deformations of them (see e.g. \cite{Ghy, Nak, AK,  Has, Rol,Win}).
Unlike complex parallelizable manifolds,  little is known about deformations of simply connected non-K\"ahler  compact complex homogeneous  manifolds.
But, we can expect they   are also  very varied and  interesting by the following observation.
Calabi-Eckmann manifolds are  simply connected non-K\"ahler  compact complex homogeneous  manifolds which are homeomorphic to the
products of two odd-dimensional spheres (\cite{CE}).
In \cite{Aka},
Akao proves that  the complete family of  deformations of a Calabi-Eckmann manifold is parametrized by matrices and he also announces that complex geometric  invariants of  deformations of a Calabi-Eckmann manifold vary depending on such parameters.

The main objects of this paper are simply connected compact Lie groups.
It is known that every  compact Lie group $K$ of even dimension admits left-invariant (or right-invariant) complex structures (\cite{Sam})  and  such complex manifold  is a complex homogeneous  manifold (\cite{Wang}).
If a compact Lie group $K$  is simply connected, then obviously it is   non-K\"ahler.
The purpose of this paper is to study the geometry of deformations of  left-invariant  complex structures on simply connected  semisimple compact Lie groups.
More precisely, we compute cohomologies of deformations of  left-invariant  complex structures which  are not left-invariant.

Let $K$ be a simply connected  semisimple compact Lie group of even dimension.
Assume that the rank is $2l$.
Let $J$ be a  left-invariant complex structure on $K$.
Then, by describing the  Kuranishi  space of   $J$, we can construct deformations $J_{\epsilon}$ of $J$ which are a priori  non-invariant.
Associated with such $J_{\epsilon}$ sufficiently near to $J$, we prove the Borel-Weil-Bott type theorem (Theorem \ref{BWBD}) and by using this we compute the cohomology of the holomorphic tangent bundle if    certain generic condition holds (Theorem \ref{HT}).
As a result, the dimension of holomorphic vector fields corresponding to $J_{\epsilon}$ is smaller than the one corresponding to any left-invariant (or right-invariant) complex structure and hence  the complex manifold $(K, J_{\epsilon})$ is not biholomorphic to $K$ with  any left-invariant or right-invariant complex structure. (Corollary \ref{nobil}).

We notice that  in \cite[Section 4]{LMN} Loeb, Manjarin and  Nicolau constructed non-invariant complex structures on a general semisimple compact Lie group $K$ of even dimension associated with  subgroups of the direct product $H\times H$ of copies of a Cartan subgroup $H$ of the complexification of $K$.
The constructions in this paper  are different from their constructions.
Our constructions are more suitable for representation theory of semi-simple Lie algebras.

We explain our computation more precisely by Lie algebras.
Let $\fk$ be the Lie algebra of $K$ of rank $2l$ and $\g=\fk_{\C}=\fk\otimes \C$.
We can associate  our non-invariant deformation $J_{\epsilon}$ with an $l$-tuple $(X_{1},\dots, X_{l})\in \g^{l}$ satisfying $[X_{i},X_{j}]=0$ for each $i,j\in \{1,\dots, l\}$.
We will prove that if an $l$-tuple $(X_{1},\dots X_{l})$ satisfies certain conditions related  to roots of a semisimple Lie algebra $\g$, then the dimension of the linear space of holomorphic vector fields associated with $J_{\epsilon}$ is equal to
\[l+\dim  \g^{\langle X_{1},\dots, X_{l}\rangle}
\]
where $\g^{\langle X_{1},\dots, X_{l}\rangle}$ is the subspace in $\g$ consisting of elements commuting with   $X_{1},\dots, X_{l}$.
This is not equal to the one associated with any left-invariant (right-invariant) complex structure.
We note that under certain geometric  conditions, we can determine  the Lie algebra structure on the space of holomorphic vector fields associated with $J_{\epsilon}$.

The case $G=SU(2)\times SU(2)$ is  a typical example. 
In this example, our results can be seen as  a Lie theoretic aspect of deformation theory  of Calabi-Eckmann  type complex structures on $S^{3}\times S^{3}$ derived from classical dynamical system as in \cite{LN}. 

\bigskip

\noindent {\bf Acknowledgement.} We are grateful to the anonymous referee for the invaluable comments and useful suggestions on improving the text.

\section{Notations and conventions}

Let $K$ be a simply connected  semisimple compact Lie group and $\fk$ the Lie algebra of $K$.
Take a maximal torus $T\subset K$ with the Lie sub-algebra $\ft$.
Take the decomposition 
$\fk=\ft\oplus \m$ associated with the Cartan-Killing bilinear form.
Let $\g=\fk_{\C}$ and $\h=\ft_{\C}$.
Then $\frak h$ is a Cartan sub-algebra of $\g$.
Denote by $\Delta$ (resp. $\Delta^{+}$) the set of roots (resp. positive roots) associated with $\h$.
Consider the subspaces $\fu=\bigoplus_{\alpha\in \Delta^{+}}\g_{\alpha}$ and $\fb=\h\oplus \bar{\fu}$ of $\g$ where $\g_{\alpha}$ is the root space of a root $\alpha$.
Then $\fb$ is a Borel sub-algebra of $\g$ and $\bar{\fu}=[\fb,\fb]$.
We have the decomposition $\g={\bar \fu}\oplus \h\oplus \fu$ and $\m=\{X+{\bar X}\mid X\in \fu\}$.
Since $K$ is simply connected, we can take a unique simply  connected complex Lie group $G$ containing $K$ as a real form.
Consider the Lie  subgroups $B$,  $U$ and  $U^{-}$ corresponding to $\fb$,   $\fu$ and  $\bar\fu$ respectively.

\section{Deformations of left-invariant complex structures}
We regard $\fk$, $\g$ etc. as the spaces of left-invariant vector fields.
Assume $\dim K$ is even. In particular, the rank $2l$ of $K$ is even. 
Let $\frak l\subset  \h$ be a subspace such that $\h=\fl\oplus \bar\fl$. 
Set $\g_{\fl}=\fl\oplus \fu$.
Then $\g_{\fl}$ is a  sub-algebra of $\g$ and $\g=\g_{\fl}\oplus \bar \g_{\fl}$ and hence we obtain a left-invariant complex structure $J_{\fl}$ on $K$.
It is known that every left-invariant complex structure can be described in this manner by taking an appropriate $\ft$.
Consider the decomposition $TK_{\C}=T^{1,0}_{J_{\fl}}K\oplus T^{0,1}_{J_{\fl}}K$ of the tangent bundle associated with the complex structure 
$J_{\fl}$.

Let $\fk^{R}$ be the space of right invariant vector fields on $K$ and $\g^{R}=\fk_{\C}^{R}$.
Denote by $\g^{R}_{{hol}}$  the image of the projection  $p^{1,0}_{J_{\fl}}:TK_{\C} \to T^{1,0}_{J_{\fl}}K$ on $\g^{R}$.
It is easily shown that $[\g^{R}, \g]=0$.
This implies that $\bar\partial \g^{R}_{hol}=0$ where  $\bar\partial$ is the Dolbeault operator on $T^{1,0}K$  and hence $\g^{R}_{hol}$ consists of holomorphic vector fields.
By $[\fl, \bar\g_{\fl}]\subset \bar\fu$, we have $\bar\partial \fl=0$.
Thus we have the spaces $\g^{R}_{hol}$ and $ \fl$ consisting  holomorphic vector fields.

Since $K$ is compact,  we obtain a holomorphic action of $G$ on $K$ so that  infinitesimal transformations of the action are holomorphic vector fields in   $\g_{hol}^{R}$ (see \cite[Section 1.8]{Ak}).
By the definition of $\g_{hol}^{R}$, for this action, the subgroup $K\subset G$ acts as the left-multiplication.
Thus this action is transitive and hence the complex manifold $K$ is identified with a quotient space $G/G_{e}$ such that $G_{e}$ is the stabilizer of the identity element $e$ of $K$.
Since $G$ and $K$ are simply connected, $G_{e}$ is the subgroup of $G$ corresponding to the subspace $\{X\in \g_{hol}^{R}\mid X_{e}=0\}$.
This space corresponds to the subspace  $ \bar\g_{\fl}\subset \g$ by the standard identification $\g^{R}\cong T_{e}K_{\C}\cong\g$.
Thus, $G_{e}$ is the group $\exp (\bar\fl)\cdot U^{-}$ where $U^{-}$ is the subgroup of $G$ corresponding to the sub-algebra $\bar \fu$. 
Since  the action of $G$ on $K= G/\exp (\bar\fl)\cdot U^{-}$ is effective, $\g_{hol}^{R}$ can be identified with the Lie algebra of $G$ and can be the space of  fundamental vector fields associated with the holomorphic action of $G$ on $G/\exp (\bar\fl)\cdot U^{-}$.
We notice that we have isomorphisms $\g\cong \g_{R}\cong \g^{R}_{hol}$. 
We can regard $  \fl$ as the space of left-invariant holomorphic vector fields on $G/\exp (\bar\fl)\cdot U^{-}$.
Since $K$ is semisimple, we have $ \g^{R}_{hol}\cap \fl=\{0\}$.

The space $\g^{R}_{hol}\oplus  \fl$ is in fact the space of all holomorphic vector fields on $K$ with the   left-invariant complex structure $J_{\fl}$ by the following proposition.
\begin{prop}\label{cohlI}
Denote by $(A^{0,*}(K, T^{1,0}_{J_{\fl}}K), \bar\partial)$  the differential graded Lie algebra of the differential forms of type $(0,\ast)$ with values in the holomorphic vector bundle $T^{1,0}_{J_{\fl}}K$ with the Dolbeault operator $ \bar\partial$.
Consider the graded Lie algebra $\bigwedge \bar\fl^{*} \otimes (\g^{R}_{hol}\oplus \fl)$ as a sub-algebra of $A^{0,*}(K, T^{1,0}_{J_{\fl}}K)$.
Then the inclusion  $\bigwedge \bar\fl^{*} \otimes (\g^{R}_{hol}\oplus \fl)\to A^{0,*}(K, T^{1,0}_{J_{\fl}}K)$ induces an isomorphism
\[\bigwedge \bar\fl^{*} \otimes (\g^{R}_{hol}\oplus \fl)\cong H^{*}(K, T^{1,0}_{J_{\fl}}K).
\]

\end{prop}
\begin{proof}
Define the decreasing filtration $F^{*}$ on $\bigwedge \bar\g^{\ast}_{\fl}$ by
\[F^{r}\left(\bigwedge \bar\g^{\ast}_{\fl}\right)=\bigoplus_{p\ge r} \bigwedge \bar\fl^{*}\otimes\bigwedge^{p}\bar\fu^{*}.
\]
By  the identification
\[A^{0,*}(K, T_{J_\fl}^{1,0}K)=\bigwedge \bar\g^{\ast}_{\fl}\otimes C^{\infty}(K)\otimes  \g_{\fl},
\]
we define the  decreasing filtration $F^{*}$ on $A^{0,*}(K, T^{1,0}_{J_{\fl}}K)$ by extending the one on  $\bigwedge \bar\g^{\ast}_{\fl}$.
Consider the spectral sequence $E_{*}^{*,*}(A^{0,*}(K, T^{1,0}_{J_{\fl}}K))$ associated with this filtration.
Then this filtration is actually the Borel spectral sequence (\cite[Appendix II]{Hir}) for the holomorphic principal fibration $K\to K/T$ where $K/T$ is regarded as a flag variety $G/B$ and $T$ is equipped with   a complex structure induced by $\ft_{\C}=\fl\oplus \bar\fl$.
Thus we have 
\[E_{2}^{p,q}(A^{0,*}(K, T^{1,0}_{J_{\fl}}K))=H^{0,q}(T)\otimes  H^{p}(G/B, E_{\g_{\fl}})
\]
where $E_{\g_{\fl}}$ is the homogeneous holomorphic vector bundle associated with the representation $ B\to GL(\g_{\fl})$ induced by the adjoint representation where we identify  $ \g_{\fl}= {\g_/\overline{ \g}_{\fl}}$.
Notice that $\fl\subset \g_{\fl}$ is a trivial $B$-submodule  and the  homogeneous holomorphic vector bundle corresponding to  the $B$-module $\fu = \g_{\fl}/\fl$ is the holomorphic tangent bundle $T^{1,0}G/B$.
We know  $H^{0,p}(G/B)=0$ for $p\ge 1$.
By Bott's result (\cite[(14.2)]{B}), we have $H^{p}(G/B, T^{1,0}G/B)=0$ for $p\ge 1$ and $H^{0}(G/B, T^{1,0}G/B)\cong \g^{R}_{hol}$ (see \cite[Section 4.8. Theorem 2]{Ak}).
We have 
\[H^{\ast}(G/B, E_{\g_{\fl}})=\g^{R}_{hol}\oplus  \fl.
\]

Consider the spectral sequence $E_{*}^{*,*}(\bigwedge \bar\fl^{*} \otimes (\g^{R}_{hol}\oplus \fl))$ of the restriction of the   filtration $F^{*}$ on $A^{0,*}(M, T^{1,0}_{J_{\fl}}K)$.
Obviously we have $E_{2}^{0,q}(\bigwedge \bar\fl^{*} \otimes (\g^{R}_{hol}\oplus \fl))=\bigwedge \bar\fl^{q} \otimes (\g^{R}_{hol}\oplus \fl)$ and $E_{2}^{p,q}(\bigwedge \bar\fl^{*} \otimes (\g^{R}_{hol}\oplus \fl))=0$ for $p\not=0$.
It is sufficient to show that the map $\iota: E_{2}^{*,*}(\bigwedge \bar\fl^{*} \otimes (\g^{R}_{hol}\oplus \fl))\to  E_{2}^{*,*}(A^{0,*}(K, T^{1,0}_{J_{\fl}}K))$ on $E_{2}$-term induced by the inclusion $\bigwedge \bar\fl^{*} \otimes (\g^{R}_{hol}\oplus \fl)\to A^{0,*}(K, T^{1,0}_{J_{\fl}}K)$ is an isomorphism. 
We can easily check that $H^{0,q}(T)\cong \bigwedge^{q} \bar\fl^{*}$.
Hence this follows from the above arguments.

\end{proof}

Let $ {\mathcal K}_{\fl}=\{ \epsilon\in \bar\fl^{*} \otimes (\g^{R}_{hol}\oplus \fl)\mid [ \epsilon, \epsilon]=0\}$.
For sufficiently small $ \epsilon\in {\mathcal K}_{\fl}$, the complex  sub-bundle $(Id+ \bar \epsilon)TK^{1,0}_{J_{\fl}}\subset TK_{\C}$ defines an integral almost complex structure $J_{\fl+\epsilon}$ on $K$.
By the above proposition,  $ \bar\fl^{*} \otimes (\g^{R}_{hol}\oplus \fl)\cong H^{1}(M, T^{1,0}_{J_{\fl}}K)$ and the cohomology class of 
the bracket  $[ \epsilon, \epsilon]$ is trivial if and only if $[ \epsilon, \epsilon]=0$.
Thus by the Kodaira-Spencer theory,  $ {\mathcal K}_{\fl}$  gives a complete deformation of $J_{\fl}$ (see \cite[Section 6]{Huy}).
In particular  every small deformation of $J_{\fl}$ is isomorphic to $J_{\fl+\epsilon}$ for some small $ \epsilon\in {\mathcal K}_{\fl}$.

Take a basis $\eta_{1},\dots \eta_{l}$ of $\bar\fl^{*}$.
Then, for  $\epsilon\in  \bar\fl^{*} \otimes (\g^{R}_{hol}\oplus \fl)$, write
\[\epsilon=\sum \eta_{i}\otimes (X_{i}+Y_{i}).
\]
for $X_{i}\in \g^{R}_{hol}$ and $Y_{i}\in  \fl$.
We have  $[ \epsilon, \epsilon]= \sum \eta_{i}\wedge \eta_{j} [X_{i},X_{j}]$. 
Thus $[ \epsilon, \epsilon]=0$ if and only if $[X_{i},X_{j}]=0$ for every $i,j\in \{1,\dots ,l\}$.
Thus $ {\mathcal K}_{\fl}$ is identified with the analytic space
\[\{(X_{1}+Y_{1},\dots, X_{l}+Y_{l})\in (\g^{R}_{hol}\oplus \fl)^{l}\mid  \forall i,j\in \{1,\dots ,l\}, [X_{i},X_{j}]=0\}.
\]

If $X_{i}=0$ for every $i\in \{1,\dots ,l\}$, then $J_{\fl+\epsilon}$ is also a left-invariant complex structure.
Otherwise $J_{\fl+\epsilon}$  a priori non-invariant.
We will see that this is essentially non-invariant.
For our main interests, we may assume $Y_{i}=0$ for every $i\in \{1,\dots ,l\}$.

We note the general fact observed by Haefliger in \cite{Ha} (see also \cite{LN}) a smooth real submanifold  in a complex  manifold which is transversal to a holomorphic foliation of complementary dimension admits a canonical complex structure.
Consider the complex homogeneous space $G/U^{-}$.
By the Iwasawa decomposition, we have the embedding $K\subset G/U^{-}$.
We notice that this embedding is not holomorphic for the complex structure $J_{\fl}$.
But the complex structure  $J_{\fl}$ is induced by the transversality  of $TK$ and the global holomorphic distribution $\bar\fl$ on $G/U^{-}$.
For  $ \epsilon\in {\mathcal K}_{\fl}$, we consider the  deformed distribution $(\Id+ \epsilon)(\bar\fl)$ on $G/U^{-}$.
If $ \epsilon\in {\mathcal K}_{\fl}$ is sufficiently  small,  $TK$ is also  transverse to $(\Id+ \epsilon)(\bar\fl)$.
This transversality  induces the deformed complex structure $J_{\fl+\epsilon}$ on $K$.
Indeed, we can check that $(\Id+ \bar \epsilon)T_{J_{\fl}}^{1,0}K$ is sent to $T^{1,0}G/U^{-}$ modulo the global  distribution  $(\Id+ \epsilon)(\bar\fl)$ via the embedding $K\subset G/U^{-}$. 
This description is inspired by \cite{LMN}.

\begin{ex}\label{CE}
Consider $K=SU(2)\times SU(2)$.
Regarding  ${\frak k}=\frak{su}(2)\oplus \frak{su}(2)$, take ${\frak k}=\langle T_{1},U_{1}, V_{1}, T_{2}, U_{2}, V_{2}\rangle$ such that 
\[T_{1}=\left( \left(
\begin{array}{ccc}
i& 0  \\
0&    -i
\end{array}
\right), 0\right), U_{1}=\left( \left(
\begin{array}{ccc}
0& 1  \\
1&    0
\end{array}
\right), 0\right),  V_{1}=\left( \left(
\begin{array}{ccc}
0& i  \\
-i&    0
\end{array}
\right), 0\right)
\]
and 
\[T_{2}=\left(0, \left(
\begin{array}{ccc}
i& 0  \\
0&    -i
\end{array}
\right)\right), U_{2}=\left( 0, \left(
\begin{array}{ccc}
0& 1  \\
1&    0
\end{array}
\right)\right),  V_{2}=\left(0, \left(
\begin{array}{ccc}
0& i  \\
-i&    0
\end{array}
\right)\right).
\]
Let $t_{1},u_{1},v_{1}, t_{2}, u_{2}, ,v_{2}$ be the dual basis. 
Take $\frak t= \langle T_{1}, T_{2}\rangle$ and $\frak u=\langle E_{2it_{1}}=U_{1}-i V_{1}, E_{2it_{2}}=U_{2}-i V_{2}\rangle$.
Then we should take $\fl=\langle -bT_{2}+i(T_{1}+aT_{2})\rangle$ for real numbers $a$ and $b\not=0$.
In this case, we have the smooth complete deformation of $J_{\fl}$ parametrized by a neighborhood  of $0$ in $\g^{R}_{hol}\oplus \fl$.

For $G=SL(2,\C)\times SL(2,\C)$, since $U^{-}$ is the lower triangular unipotent subgroup of $G$,  we identify $G/U^{-}\cong (\C^{2}-\{0\})\times (\C^{2}-\{0\})$ by 
\[G/U^{-}\ni \left[ \left(
\begin{array}{ccc}
z_{1}^{\prime}& z_{1}  \\
z_{2}^{\prime}&    z_{2}
\end{array}
\right),  \left(
\begin{array}{ccc}
w_{1}^{\prime}& w_{1}  \\
w_{2}^{\prime}&    w_{2}
\end{array}
\right)\right]\mapsto ((z_{1},z_{2}),(w_{1},w_{2}))\in (\C^{2}-\{0\})\times (\C^{2}-\{0\}).
\]
In this identification, we can regard $\bar\fl=\left\langle z_{1}\frac{\partial}{\partial z_{1}}+z_{2}\frac{\partial}{\partial z_{2}}+c(w_{1}\frac{\partial}{\partial w_{1}}+w_{2}\frac{\partial}{\partial w_{2}})\right\rangle$ where $c=a+bi$.
We can  see that   $SU(2)$ is embedded in $\C^{2}-\{0\}$ as the standard $3$-sphere $S^{3}\subset \C^{2}-\{0\}$.
Thus, the complex manifold $(K,J_{\fl})$ is $S^{3}\times S^{3}$ with the complex structure given by the transversality between  $S^{3}\times S^{3}$ and the global distribution $\left\langle z_{1}\frac{\partial}{\partial z_{1}}+z_{2}\frac{\partial}{\partial z_{2}}+c(w_{1}\frac{\partial}{\partial w_{1}}+w_{2}\frac{\partial}{\partial w_{2}})\right\rangle$.
This is actually a Calabi-Eckmann manifold (see \cite{LN}).  
We  have
\[\g^{R}_{hol}=\left\{\sum_{i,j\le 2}a_{ij} z_{i}\frac{\partial}{\partial z_{j}}\relmiddle| a_{11}+a_{22}=0\right\}  \oplus \left\{\sum_{i,j\le 2}b_{ij} w_{i}\frac{\partial}{\partial w_{j}}\relmiddle| b_{11}+b_{22}=0\right\}\]
on $G/U^{-}\cong (\C^{2}-\{0\})\times (\C^{2}-\{0\})$.
This implies
\[\g^{R}_{hol}\oplus \fl=\left\{\sum_{i,j\le 2}\left(a_{ij} z_{i}\frac{\partial}{\partial z_{j}}+b_{ij} w_{i}\frac{\partial}{\partial w_{j}}\right)\relmiddle| a_{11}+a_{22}+cb_{11}+cb_{22}=0\right\} .
\]
By the above argument, every small deformation of  $J_{\fl}$ is  given by the transversality between   $S^{3}\times S^{3}$ and the global distribution $\left\langle z_{1}\frac{\partial}{\partial z_{1}}+z_{2}\frac{\partial}{\partial z_{2}}+c(w_{1}\frac{\partial}{\partial w_{1}}+w_{2}\frac{\partial}{\partial w_{2}})+\sum_{i,j\le 2}\left(a_{ij} z_{i}\frac{\partial}{\partial z_{j}}+b_{ij} w_{i}\frac{\partial}{\partial w_{j}}\right)\right\rangle$ for complex numbers $a_{ij}, b_{ij}$ with $a_{11}+a_{22}+cb_{11}+cb_{22}=0$.
This fact is also a consequence of  the result in \cite{LN}.

\end{ex}

\section{Cohomology computation for $J_{\fl+\epsilon}$ }

In this section we study the cohomology of      $K$ with a (a priori) non-left-invariant deformed complex structure $ J_{\fl+\epsilon}$ for $\epsilon=\sum \eta_{i}\otimes X_{i}$ such that  $X_{i}\in \g^{R}_{hol}$ satisfy $[X_{i},X_{j}]=0$ for every $i,j\in \{1,\dots ,l\}$.
We fix  a basis $\eta_{1},\dots, \eta_{l}$ of $\bar\fl^{*}$ and consider $X_{1},\dots ,X_{l}$ as parameters of deformations.

\subsection{Kostant's theorem}
Consider  the decompositions $\g={\bar \fu}\oplus \h\oplus \fu$ and 
$\fu=\bigoplus_{\alpha\in \Delta^{+}}\g_{\alpha}$.
Denote by $(,)_{CK}:\g\times\g\to \C$ the Cartan-Killing form.
Then we can choose a basis $\{E_{\alpha}\}_{\alpha\in \Delta^{+}}$  of $\fu$ (resp. $\{E_{-\alpha}\}_{\alpha\in \Delta^{+}}$  of ${\bar \fu}$) so that
\begin{itemize}
\item $E_{\alpha}\in \g_{\alpha}$, 
\item $(E_{\alpha},E_{\beta})_{CK}=\delta_{\alpha\beta}$ for $\alpha,\beta\in \Delta^{+}$,
\item $([E_{\alpha},E_{-\alpha}], A)_{CK}=\alpha(A) $ for $A\in \h$.

\end{itemize}

Identifying   $\fu$ as the  dual space ${\bar \fu}^{\ast}$   via the  Cartan-Killing form,  we can regard $\{E_{\alpha}\}_{\alpha\in \Delta^{+}}$ as a basis of ${\bar \fu}^{\ast}$.
For each (ordered) subset $\Phi=\{\alpha_{1},\dots, \alpha_{j}\}\subset \Delta^{+}$, we write $E_{\Phi}=E_{\alpha_{1}}\wedge \dots \wedge E_{\alpha_{j}}$ as an element of  $\bigwedge {\bar \fu}^{\ast}$.
Let $W$ be the Weyl  group of $\g$.
For $\sigma\in W$, we define the subset $\Phi_{\sigma}\subset  \Delta^{+}$ by $\Phi_{\sigma}=\sigma (-\Delta^{+})\cap  \Delta^{+}$.
The number $\vert \Phi_{\sigma}\vert$ is called the length of $\sigma\in W$.
For a simple root $\alpha$, denoting by $\sigma_{\alpha}\in W$ the reflection corresponding to $\alpha$,
we note $\sigma_{\alpha} (-\Delta^{+})\cap  \Delta^{+}=\{\alpha\}$.
Denote by $W(k) \subset W$ the subset of  elements of  length $k$.

Let $D$ be the set of dominant integral weights.
For $\lambda\in D$,  let $\nu_{\lambda}: \g\to {\rm End}(V^{\lambda})$ be the irreducible representation of $\g$ on a vector space $V^{\lambda}$ whose highest weight is $\lambda$.
We consider the dual representation $\nu^{\ast}_{\lambda}: \g\to {\rm End}((V^{\lambda})^{\ast})$.
This is an irreducible representation whose lowest weight is $-\lambda$.
For each $\sigma\in W$,  take a weight vector $v_{\pm \sigma \lambda}$ of $(V^{\lambda})^{\ast}$  for the weight $\pm \sigma\lambda$ which exists uniquely up to scalar multiplication.

Consider the cochain complex $\bigwedge {\bar \fu}^{\ast}\otimes (V^{\lambda})^{\ast}$  of the Lie algebra ${\bar \fu}$ for the restriction of the representation $\nu_{\lambda}$.
For the cohomology $H^{\ast}({\bar \fu}, (V^{\lambda})^{\ast})$ of this complex, we have:
\begin{thm}[{\cite[Theorem 5.14]{Kos}}]\label{Kost}
\[H^{k}({\bar \fu}, (V^{\lambda})^{\ast})=\bigoplus_{\sigma\in W(k)}\langle [ E_{\Phi_{\sigma}}\otimes v_{-\sigma \lambda}]\rangle.
\]
\end{thm}

\begin{rem}
In \cite{Kos}, Kostant stated $H^{k}( \fu, V^{\lambda})=\bigoplus_{\sigma\in W(k)}\langle [ E_{-\Phi_{\sigma}}\otimes v_{\sigma \lambda}]\rangle$.
\end{rem}

\subsection{Deformed Borel-Weil-Bott Theorem}
We consider a deformed complex manifold $(K, J_{\fl+\epsilon})$ for $\epsilon=\sum \eta_{i}\otimes X_{i}$ such that  $X_{i}\in \g^{R}_{hol}$ satisfy $[X_{i},X_{j}]=0$ for every $i,j\in \{1,\dots ,l\}$.
Take a dual basis $A_{1},\dots, A_{l}$ of  $\eta_{1},\dots \eta_{l}$.
Then, since $\epsilon$ is sufficiently small, the holomorphic tangent bundle $T^{(1,0)}_{J_{\fl+\epsilon}}K$ of $(K, J_{\fl+\epsilon})$ is $C^{\infty}$-trivialized (not holomorphic) by the global frame $\{ A_{1}+\bar{X}_{1},\dots, A_{l}+\bar{X}_{l},\{E_{\alpha}\}_{\alpha\in \Delta^{+}}\}$.
Since $X_{1},\dots, X_{l}$ are holomorphic for the left-invariant complex structure $J_{\fl}$ and $[X_{i},X_{j}]=0$ for every $i,j\in \{1,\dots ,l\}$, we have
\[[A_{i}+\bar{X}_{i},A_{j}+\bar{X}_{j}]=0
\]
and 
\[[A_{i}+\bar{X}_{i}, E_{\alpha}]=\alpha(A_{i})E_{\alpha}.
\]
Thus the vector space $\g_{\fl}^{\epsilon}=\langle A_{1}+\bar{X}_{1},\dots A_{l}+\bar{X}_{l},\{E_{\alpha}\}\rangle$ is a Lie sub-algebra of the Lie algebra of vector fields on $K$ and we have the isomorphism $\g_{\fl}^{\epsilon}\cong \g_{\fl}$ given by
\[\g_{\fl}^{\epsilon}\ni A_{i}+\bar{X}_{i}\mapsto A_{i}\in \g_{\fl}
\]
and 
\[\g_{\fl}^{\epsilon}\ni E_{\alpha}\mapsto E_{\alpha}\in \g_{\fl}.
\]
For the abelian Lie algebra $\fl^{\epsilon}=\langle A_{1}+\bar{X}_{1},\dots A_{l}+\bar{X}_{l}\rangle$, we have
$\g_{\fl}^{\epsilon}=\fl^{\epsilon}\ltimes \fu$.

Let $\rho:  \bar{\g_{\fl}}\to \gl(V_{\rho})$ be a representation on a complex vector space $V_{\rho}$.
Via the isomorphism $\g_{\fl}^{\epsilon}\cong \g_{\fl}$ we obtain the representation $\rho_{\epsilon}:\bar{\g_{\fl}^{\epsilon}}\to \gl(V_{\rho})$.
We consider the $C^{\infty}$-trivial vector bundle ${\mathcal E}_{\rho}=K\times V_{\rho}$ with the holomorphic structure $\bar\partial _{\rho}=\bar\partial+ \rho_{\epsilon}$ where we regard $\rho_{\epsilon}\in \bar{\g_{\fl}^{\epsilon}}^{\ast}\otimes  \gl(V_{\rho})$ as a $(0,1)$-form on $K$ with values in ${\rm End}({\mathcal E}_{\rho})$.
We will compute the cohomology of the Dolbeault complex $(A^{0,\ast}(K, E_{\rho}),\bar\partial _{\rho})$ with values in the holomorphic  vector bundle ${\mathcal E}_{\rho}$.
Regrading the space $C^{\infty}(K)$ of smooth functions on $K$ as a $\bar{\g_{\fl}^{\epsilon}}$-module, 
we identify  the Dolbeault complex $(A^{0,\ast}(K, {\mathcal E}_{\rho}),\bar\partial _{\rho})$ with the Lie algebra complex
\[\bigwedge \bar{\g_{\fl}^{\epsilon}}^{\ast}\otimes V_{\rho}\otimes C^{\infty}(K).
\]
By the Peter-Weyl theorem, we have the decomposition $L^{2}(K)=\sum_{\lambda\in D} V^{\lambda}\otimes (V^{\lambda })^{\ast}$ as a $\fk^{R}\oplus \fk$-module.
For each component  $V^{\lambda}\otimes (V^{\lambda })^{\ast}$, $V^{\lambda}$ can be considered as a $\g^{R}_{hol}$-module by the following way.
Each $V^{\lambda}\otimes (V^{\lambda })^{\ast}$ consists of matrix elements of a $K$-representation $V^{\lambda}$.
By the unitary trick,  $V^{\lambda}$ is canonically extended to  a holomorphic representation of $G$.
It is known that the the inclusion $K\subset G$ induces the correspondence between  matrix elements of a holomorphic  $G$-representation $V^{\lambda}$ and matrix elements of a $K$-representation $V^{\lambda}$ (see \cite[Section 5.3]{Ak}).
Thus, for each component  $V^{\lambda}\otimes (V^{\lambda })^{\ast}$, $V^{\lambda}$ can be considered as a  representation of the Lie algebra of the right invariant holomorphic vector fields on $G$ which is identified with $\g^{R}_{hol}$ via  the holomorphic action of $G$ on $K$ with the  complex structure $J_{\fl}$ as in the last section.

Take the Hermitian  metric of $T^{(1,0)}_{J_{\fl+\epsilon}}K$ so that the global frame $\{ A_{1}+\bar{X}_{1},\dots ,A_{l}+\bar{X}_{l},\{E_{\alpha}\}_{\alpha\in \Delta^{+}}\}$ is normal orthogonal.
Then, its Laplacian operator preserves each  $\bigwedge\bar{\g_{\fl}^{\epsilon}}^{\ast}\otimes V_{\rho}\otimes V^{\lambda}\otimes (V^{\lambda })^{\ast}$.
Thus, we have an isomorphism
\[H^{\ast}(\bar{\g_{\fl}^{\epsilon}}, V_{\rho}\otimes C^{\infty}(K))\cong \bigoplus_{\lambda\in D} H^{\ast}(\bar{\g_{\fl}^{\epsilon}}, V_{\rho}\otimes V^{\lambda}\otimes (V^{\lambda })^{\ast}).
\]
Since we must regard the $\fk^{R}$-module $V^{\lambda}$ as a $\g^{R}_{hol}$-module as above, 
we notice that  $V^{\lambda}$  is a non-trivial $\bar{\g_{\fl}^{\epsilon}}$-module as $(\bar{A}_{i}+{X}_{i})\cdot v=\nu_{\lambda}(X_{i})v$ for $v\in V^{\lambda}$.
Thus we should compute the cohomology of the Lie algebra complex $\bigwedge \bar{\g_{\fl}^{\epsilon}}^{\ast}\otimes V_{\rho}\otimes V^{\lambda}\otimes (V^{\lambda })^{\ast}$.
We consider the spectral sequence $E^{p,q}_{2}\left (\bigwedge \bar{\g_{\fl}^{\epsilon}}^{\ast}\otimes V_{\rho}\otimes V^{\lambda}\otimes (V^{\lambda })^{\ast} \right)$ of this complex associated with the ideal $\bar{\fu}\subset \bar{\g_{\fl}^{\epsilon}}$.
Then we have
\[E^{p,q}_{2}\left (\bigwedge \bar{\g_{\fl}^{\epsilon}}^{\ast}\otimes V_{\rho}\otimes V^{\lambda}\otimes (V^{\lambda })^{\ast} \right)=
H^{p}(\bar{\fl^{\epsilon}}, H^{q}(\bar{\fu}, V_{\rho}\otimes (V^{\lambda })^{\ast})\otimes V^{\lambda}).
\]
We suppose $\dim V_{\rho}=1$. Then $\rho$ is a character of $\bar{\g_{\fl}}$.
Thus $\rho_{\epsilon}$ is a character of the abelian Lie algebra $\bar{\fl^{\epsilon}}$ and
we have 
\[E^{p,q}_{2}\left (\bigwedge \bar{\g_{\fl}^{\epsilon}}^{\ast}\otimes V_{\rho}\otimes V^{\lambda}\otimes (V^{\lambda })^{\ast} \right)=
H^{p}(\bar{\fl^{\epsilon}}, H^{q}(\bar{\fu},  (V^{\lambda })^{\ast})\otimes V^{\lambda}\otimes V_{\rho}).
\]
By Theorem \ref{Kost}, we have 
\[E^{p,q}_{2}\left (\bigwedge \bar{\g_{\fl}^{\epsilon}}^{\ast}\otimes V_{\rho}\otimes V^{\lambda}\otimes (V^{\lambda })^{\ast} \right)=
H^{p}\left(\bar{\fl^{\epsilon}}, \bigoplus_{\sigma\in W(q)}\langle [ E_{\Phi_{\sigma}}\otimes v_{-\sigma \lambda}]\rangle\otimes V^{\lambda}\otimes V_{\rho}\right).
\]
Since $[X_{i},X_{j}]=0$ for every $i,j\in \{1,\dots ,l\}$, there exists a set $S(\lambda,X_{1},\dots, X_{l})= \{\beta\}$ of maps  $\beta: \{X_{1},\dots X_{l}\}\to \C$ such that we can decompose $V^{\lambda}=\bigoplus_{\beta \in S(\lambda,X_{1},\dots, X_{l})} V^{\lambda}(\beta)$ so that for every $i\in \{1,\dots ,l\}$,  $V^{\lambda}(\beta)$ is the generalized $\beta(X_{i})$-eigenspace of $\nu^{\lambda}(X_{i})$.

\begin{defi}\label{reso}
We say $(\sigma, \lambda, \beta)$ is a $(\rho, X_{1},\dots, X_{l})$-{\em resonance} if $\sigma\in W$, $\lambda\in D$ and $\beta \in S(\lambda,X_{1},\dots, X_{l})$ satisfy the equation
\[\sum_{\alpha\in \Phi_{\sigma}}\alpha(\bar{A}_{i})-\sigma\lambda(\bar{A}_{i})+\rho(\bar{A}_{i})+\beta(X_{i})=0
\]
for every $i\in \{1,\dots ,l\}$.
\end{defi}

\begin{rem}
This definition is inspired  by the resonant  condition in the  Poincar\'e-Dulac theorems  and its application to the deformations of complex structures (\cite{Ha, LN}).
\end{rem}

Denote by $R(\rho, X_{1},\dots, X_{l})$ the set of  $(\rho, X_{1},\dots, X_{l})$-resonances and denote by $R(\rho, X_{1},\dots, X_{l})(\lambda)$ its subspace $\{(\sigma, \lambda, \beta)\in R(\rho, X_{1},\dots, X_{l})\}$ for fixed $\lambda$.
We consider the subspace 
\[ \bigwedge \bar{\fl^{\epsilon}}^{\ast}\otimes  \bigoplus_{R(\rho, X_{1},\dots, X_{l})(\lambda)} \langle E_{\Phi_{\sigma}}\otimes v_{-\sigma \lambda}\rangle \otimes V^{\lambda}(\beta)\otimes V_{\rho}
\]
of $\bigwedge \bar{\g_{\fl}^{\epsilon}}^{\ast}\otimes V_{\rho}\otimes V^{\lambda}\otimes (V^{\lambda })^{\ast}=\bigwedge \bar{\fl^{\epsilon}}^{\ast}\otimes\bigwedge \bar{\fu}^{\ast} \otimes V_{\rho}\otimes V^{\lambda}\otimes (V^{\lambda })^{\ast}$. 
We can say that this subspace is a sub-complex by the following way.
Define the nilpotent $\bar{\fl^{\epsilon}}$-module $\tilde{V}^{\lambda}(\beta)$ so that $\tilde{V}^{\lambda}(\beta)=V^{\lambda}(\beta)$ as a vector space and $(\bar{A}_{i}+{X}_{i})\cdot v=(\nu_{\lambda}(X_{i})-\beta(X_{i}))v$ for $v\in \tilde{V}(\beta)$.
Then each component
\[ \bigwedge \bar{\fl^{\epsilon}}^{\ast}\otimes  \langle E_{\Phi_{\sigma}}\otimes v_{-\sigma \lambda}\rangle \otimes V^{\lambda}(\beta)\otimes V_{\rho}
\]
is identified with the Lie algebra complex  $\bigwedge \bar{\fl^{\epsilon}}^{\ast}\otimes \tilde{V}^{\lambda}(\beta)$ by the equation of the Definition \ref{reso} shifting the degree by $-\vert \Phi_{\sigma}\vert$.
\begin{prop}
The inclusion 
\[\bigwedge \bar{\fl^{\epsilon}}^{\ast}\otimes  \bigoplus_{(\sigma, \lambda, \beta) \in R(\rho, X_{1},\dots, X_{l})(\lambda)} \langle E_{\Phi_{\sigma}}\otimes v_{-\sigma \lambda}\rangle \otimes V^{\lambda}(\beta)\otimes V_{\rho}\subset \bigwedge \bar{\g_{\fl}^{\epsilon}}^{\ast}\otimes V_{\rho}\otimes V^{\lambda}\otimes (V^{\lambda })^{\ast}
\]
induces a cohomology isomorphism.
\end{prop}
\begin{proof}
Consider the spectral sequence of $\bigwedge \bar{\fl^{\epsilon}}^{\ast}\otimes  \bigoplus_{R(\rho, X_{1},\dots, X_{l})(\lambda)} \langle E_{\Phi_{\sigma}}\otimes v_{-\sigma \lambda}\rangle \otimes V^{\lambda}(\beta)\otimes V_{\rho}$ for the filtration which is the restriction of the filtration on $\bigwedge \bar{\g_{\fl}^{\epsilon}}^{\ast}\otimes V_{\rho}\otimes V^{\lambda}\otimes (V^{\lambda })^{\ast}$ associated with the ideal $\bar{\fu}\subset \bar{\g_{\fl}^{\epsilon}}$.
Then its $E_{2}$-term is
\[H^{p}\left(\bar{\fl^{\epsilon}}, \bigoplus_{\substack{R(\rho, X_{1},\dots, X_{l})(\lambda)\\ \sigma\in W(q)}}\langle [ E_{\Phi_{\sigma}}\otimes v_{-\sigma \lambda}]\rangle\otimes V^{\lambda}\otimes V_{\rho}\right).
\]
By simple arguments on the cohomology of abelian Lie algebra(see \cite[Lemma 3.1]{Kasd}), we can  say that this $E_{2}$-term is isomorphic to 
\[E^{p,q}_{2}\left (\bigwedge \bar{\g_{\fl}^{\epsilon}}^{\ast}\otimes V_{\rho}\otimes V^{\lambda}\otimes (V^{\lambda })^{\ast} \right)=
H^{p}\left(\bar{\fl^{\epsilon}}, \bigoplus_{\sigma\in W(q)}\langle [ E_{\Phi_{\sigma}}\otimes v_{-\sigma \lambda}]\rangle\otimes V^{\lambda}\otimes V_{\rho}\right).
\]
This implies that  the inclusion we  should consider induces an isomorphism on the $E_{2}$-term of the spectral sequences and hence 
 the proposition follows.
\end{proof}

This proposition gives the following Borel-Weil-Bott type theorem.

\begin{thm}\label{BWBD}
\[H^{\ast}(K, {\mathcal E}_{\rho})\cong \bigoplus_{(\sigma,\lambda,\beta)\in R(\rho, X_{1},\dots, X_{l})} H^{\ast-\vert \Phi_{\sigma}\vert}(\bar{\fl^{\epsilon}}, \langle E_{\Phi_{\sigma}}\otimes v_{-\sigma \lambda}\rangle\otimes V^{\lambda}(\beta)\otimes V_{\rho}).
\]
In particular, if $R(\rho, X_{1},\dots, X_{l})=\o$, then $H^{\ast}(K, {\mathcal E}_{\rho})=0$. 
\end{thm}
For the identity element $e\in W$, define the subset $R_{e}(\rho, X_{1},\dots, X_{l})=\{(e, \lambda, \beta)\in R(\rho, X_{1},\dots, X_{l})\}$ of $R(\rho, X_{1},\dots, X_{l})$.
Then:
\begin{cor}
\[ H^{0}(K, {\mathcal E}_{\rho})=\bigoplus_{(\sigma,\lambda,\beta)\in R_{e}(\rho, X_{1},\dots, X_{l})}\left(\langle  v_{-\lambda}\rangle\otimes  V^{\lambda}(\beta)\otimes V_{\rho}\right)^{\langle X_{1},\dots, X_{l}\rangle}.\]

\end{cor}

If the set $\{X_{1},\dots, X_{l}\}$ is contained in some Cartan sub-algebra of $\g^{R}_{hol}$, then all $\nu^{\lambda}(X_{i})$ are simultaneously diagonalizable and so each nilpotent $\bar{\fl^{\epsilon}}$-module $\tilde{V}(\beta)$ is trivial.
Thus in this case we have
\[H^{\ast}(K, {\mathcal E}_{\rho})\cong \bigoplus_{(\sigma,\lambda,\beta)\in R(\rho, X_{1},\dots, X_{l})} \bigwedge \bar{\fl^{\epsilon}}^{\ast}\otimes \langle E_{\Phi_{\sigma}}\otimes v_{-\sigma \lambda}\rangle\otimes  V^{\lambda}(\beta)\otimes V_{\rho}.
\]

\begin{defi}
$\{X_{1},\dots, X_{l}\}$ is called non-resonant  for $\rho$ if $R(\rho, X_{1},\dots, X_{l})=\emptyset $ or $R(\rho, X_{1},\dots, X_{l})=\{(\sigma, \lambda, 0)\}$.
\end{defi}

\begin{rem}\label{regen}
Suppose $\rho$ is   integral.
Then we can say that  the  non-resonant condition is generic in the following sense.
Since  $\sum_{\alpha\in \Phi_{\sigma}}\alpha-\sigma\lambda+\rho$ is integral for  any $\sigma\in W$ and $\lambda\in D$,
we can take  $A_{1},\dots, A_{l}$ so that  $\sum_{\alpha\in \Phi_{\sigma}}\alpha(\bar{A}_{i})-\sigma\lambda(\bar{A}_{i})+\rho(\bar{A}_{i})\in\Q[\sqrt{-1}]$ for  all $\sigma, \lambda $.
Let $w_{1},\dots ,w_{r}$ be the fundamental weights of $\g$ associated with $\h$.
Take $V^{w_{i}}=\bigoplus_{\beta_{ij}\in S(w_{i},X_{1},\dots, X_{l})} V^{\lambda}(\beta_{ij})$ as above.
Since any $\lambda\in D$ is $\lambda=\sum n_{i} w_{i}$ ($n_{i}\in \N$), 
each $\beta\in S(\lambda ,X_{1},\dots, X_{l})$ is $\beta=\sum m_{ij}\beta_{ij}$ ($m_{ij}\in \N$).
If  $\beta_{ij}(X_{k})$ is irrational or $0$ for every $i,j,k$, then $\{X_{1},\dots, X_{l}\}$ is  non-resonant for $\rho$ since we can say that   the equality  
\[\sum_{\alpha\in \Phi_{\sigma}}\alpha(\bar{A}_{i})-\sigma\lambda(\bar{A}_{i})+\rho(\bar{A}_{i})+\beta(X_{i})=0
\] holds for every $i\in \{1,\dots ,l\}$ only if $\beta=0$.

The condition $R(\rho, X_{1},\dots, X_{l})=\{(\sigma, \lambda, 0)\}$ occurs only if $\sum_{\alpha\in \Phi_{\sigma}}\alpha-\sigma\lambda+\rho=0$ for some $\sigma\in W$ and $\lambda\in D$.
We notice that $\sum_{\alpha\in \Phi_{\sigma}}\alpha-\sigma\lambda+\rho=0$ for some $\sigma\in W$ and $\lambda\in D$ if and only if $\rho$ is integral  and   $\rho+\frac{1}{2}\sum_{\alpha \in \Delta^{+}}\alpha$ is regular in the sense of \cite[Section 5.9]{Kos}.
By Theorem \ref{BWBD}, if $\{X_{1},\dots, X_{l}\}$ is  non-resonant  for $\rho$, then we have
\[ H^{\ast}(K, {\mathcal E}_{\rho})\cong  H^{\ast-\vert \Phi_{\sigma}\vert}(\bar{\fl^{\epsilon}}, \langle E_{\Phi_{\sigma}}\otimes v_{-\sigma \lambda}\rangle\otimes V^{\lambda}(0)\otimes V_{\rho})
\] for some $\sigma\in W$
if  $\rho$ is integral  and   $\rho+\frac{1}{2}\sum_{\alpha \in \Delta^{+}}\alpha$ is regular or 
\[ H^{\ast}(K, {\mathcal E}_{\rho})\cong 0\]
otherwise.
We notice   that we have
\[H^{\ast}(\bar{\fl^{\epsilon}}, \langle E_{\Phi_{\sigma}}\otimes v_{-\sigma \lambda}\rangle\otimes V^{\lambda}(0)\otimes V_{\rho}) \cong H^{\ast}(\C^{l}, V^{\lambda}(0))\]
 by the equality $\sum_{\alpha\in \Phi_{\sigma}}\alpha-\sigma\lambda+\rho=0$  and  $\dim  \langle E_{\Phi_{\sigma}}\otimes v_{-\sigma \lambda}\rangle\otimes V_{\rho}= 1$
where $V^{\lambda}(0)$ is the  generalized eigenspace corresponding to the eigenvalue $0$ for the linear  operators  $\nu^{\lambda}(X_{1}),\dots, \nu^{\lambda}(X_{l})$ and  $H^{\ast}(\C^{l}, V^{\lambda}(0))$ is the Lie algebra cohomology of the abelian Lie algebra  $\C^{l}$ with values in the representation $\C^{l}\to {\rm End}(V^{\lambda}(0))$ so that for the standard basis $e_{1},\dots e_{l}$ of  $\C^{l}$, each $e_{i}$ represents the linear  operator of $\nu^{\lambda}(X_{i})$. 
\end{rem}

Obviously $\{0,\dots, 0\}$ is  non-resonant  for any $\rho$, in this case
\[ H^{\ast}(K, {\mathcal E}_{\rho})\cong 0\]
or 
\[ H^{\ast}(K, {\mathcal E}_{\rho})\cong \bigwedge \bar{\fl^{\epsilon}}^{\ast}\otimes V^{\lambda}.
\]
This is a consequence of Borel-Weil-Bott Theorem as in \cite{B}, \cite[Section 6]{Kos}.

\subsection{Cohomology of the holomorphic tangent bundle}
We consider the holomorphic tangent bundle $T^{(1,0)}_{J_{\fl+\epsilon}} K$ of the complex manifold $(K, J_{\fl+\epsilon})$.
This is the vector bundle generated by global ${\mathcal C}^{\infty}$-distribution  $\g_{\fl}^{\epsilon}=\langle A_{1}+\bar{X}_{1},\dots A_{l}+\bar{X}_{l},\{E_{\alpha}\}\rangle$.
Let $p^{(1,0)}_{\epsilon}:TK_{\C}\to T^{(1,0)}_{J_{\fl+\epsilon}} K$ be the projection.
Then we consider the space $\h^{(1,0)}_{\epsilon}=p^{(1,0)}_{\epsilon}(\h)$ of sections of $T^{(1,0)}_{J_{\fl+\epsilon}} K$.
By  $\ft_{\C}=\h$,  $\h^{(1,0)}_{\epsilon}=\{B-\sqrt{-1}J_{\fl+\epsilon}(B)\mid B\in \ft\}$.
We can take sufficiently small $\epsilon$ so that $\dim \fl=\dim \h^{(1,0)}_{\epsilon}$.
Since we have $[B, E_{-\alpha}]=-\alpha(B)E_{-\alpha}$ and $[B, A_{i}+\bar{X}_{i}]=0$ for any $B\in \h$,
we can say  $\bar\partial (\h^{(1,0)}_{\epsilon})=0$ where $\bar\partial$ is the Dolbeault operator of $T^{(1,0)}_{J_{\fl+\epsilon}} K$.
By $[E_{\alpha}, E_{-\alpha}]\in \sqrt{-1}\ft$, we have $p^{(1,0)}_{\epsilon}([ E_{\alpha} ,E_{-\alpha}])\in \h^{(1,0)}_{\epsilon}$.
This implies that the vector space $V_{\tau}$ of sections of $T^{(1,0)}_{J_{\fl+\epsilon}} K$ generated by $\h^{(1,0)}_{\epsilon}$ and $\{E_{\alpha}\}$ is a $\bar{\g_{\fl}^{\epsilon}}$-module with a representation $\tau_{\epsilon}:\bar{\g_{\fl}^{\epsilon}}\to {\rm End}(V_{\tau})$.
For the representation $\tau:\bar{\g_{\fl}}\to {\rm End}(V_{\tau})$ associated with $\tau_{\epsilon}$ via the isomorphism  $\g_{\fl}^{\epsilon}\cong \g_{\fl}$, we have $T^{(1,0)}_{J_{\fl+\epsilon}} K\cong {\mathcal E}_{\tau}$.
We compute the cohomology $H^{\ast}(K,T^{(1,0)}_{J_{\fl+\epsilon}} K)=H^{\ast}(K,{\mathcal E}_{\tau})$.

Let ${\mathcal E}_{\h^{(1,0)}_{\epsilon}}\subset {\mathcal E}_{\tau}$ be the sub-bundle of ${\mathcal E}_{\tau}$  generated by the distribution $\h^{(1,0)}_{\epsilon}$.
Since $\bar\partial (\h^{(1,0)}_{\epsilon})=0$, ${\mathcal E}_{\h^{(1,0)}_{\epsilon}}$ is a trivial holomorphic bundle.
Thus, we have 
\[H^{\ast}(K,{\mathcal E}_{\h^{(1,0)}_{\epsilon}})\cong H^{0,\ast}(K)\otimes \h^{(1,0)}_{\epsilon}.
\]
If $\{X_{1},\dots, X_{l}\}$ is  non-resonant  for the trivial representation $\rho=0$,
we have
\[H^{0,\ast}(K)\cong \bigwedge \bar{\fl^{\epsilon}}^{\ast}.
\]
We consider the inclusion $\bigwedge \bar{\fl^{\epsilon}}^{\ast}\otimes \h^{(1,0)}_{\epsilon}\subset A^{0,*}(M, T^{1,0}_{J_{\fl+\epsilon}}K)$.
Then, taking the Hermitian metric on $T^{1,0}_{J_{\fl+\epsilon}}K$ associated with the global frame $A_{1}+\bar{X}_{1},\dots A_{l}+\bar{X}_{l},\{E_{\alpha}\}$, we can say that $\bigwedge \bar{\fl^{\epsilon}}^{\ast}\otimes \h^{(1,0)}_{\epsilon}$ consists harmonic forms in $A^{0,*}(M, T^{1,0}_{J_{\fl+\epsilon}}K)$.
Thus  the inclusion $\bigwedge \bar{\fl^{\epsilon}}^{\ast}\otimes \h^{(1,0)}_{\epsilon}\subset A^{0,*}(M, T^{1,0}_{J_{\fl+\epsilon}}K)$ induces an injection $H^{\ast}(K,{\mathcal E}_{\h^{(1,0)}_{\epsilon}})\cong \bigwedge \bar{\fl^{\epsilon}}^{\ast}\otimes \h^{(1,0)}_{\epsilon} \hookrightarrow H^{\ast}(K,T^{(1,0)}_{J_{\fl+\epsilon}} K)$.

We consider the quotient bundle ${\mathcal E}_{\fu}={\mathcal E}_{\tau}/{\mathcal E}_{\h^{(1,0)}_{\epsilon}}$.
We have $H^{\ast}(K, {\mathcal E}_{\tau})\cong  \bigwedge \bar{\fl^{\epsilon}}^{\ast}\otimes \h^{(1,0)}_{\epsilon}\oplus H^{\ast}(K, {\mathcal E}_{\fu})$ by the above argument.
We decompose $\Delta=\Delta_{1}\sqcup\dots, \sqcup\Delta_{m}$ such that each $\Delta_{i}$ is an irreducible root system.
Then, we have a direct sum ${\mathcal E}_{\fu}={\mathcal E}_{1}\oplus\dots \oplus {\mathcal E}_{m}$ where  each ${\mathcal E}_{i}$ is the sub-bundle of ${\mathcal E}_{\fu}$ generated by $\{E_{\alpha}\}_{\alpha\in \Delta_{i}\cap \Delta^{+}}$.

We consider the irreducible root system $\Delta_{i}$.
For each $\alpha\in \Delta_{i}\cap \Delta^{+}$, we denote by ${\mathcal V}_{\alpha}$  the smooth sub-bundle of $ {\mathcal E}_{i}$ generated by the vector field $E_{\alpha}$.
Then, by the theory of roots, we can take a total ordering $\le$ of $\Delta_{i}\cap \Delta^{+}$ such that 
for any $\beta\in \Delta_{i}\cap \Delta^{+}$ the smooth sub-bundle ${\mathcal E}_{\le \beta}=\bigoplus_{\alpha\le \beta}{\mathcal V}_{\alpha}$ of $ {\mathcal E}_{i}$ is a holomorphic sub-bundle  of $ {\mathcal E}_{i}$ and the maximal one in $ \Delta_{i}\cap \Delta^{+}$ associated with this order is a highest weight of an adjoint representation  of $\g$.
Writing $ \Delta_{i}\cap \Delta^{+}=\{\alpha_{1},\dots ,\alpha_{M}\}$ so that $\alpha_{j}\le \alpha_{k}$ if and only if $j\le k$,
then we have a sequence 
\[ {\mathcal E}_{\le \alpha_{1}}\subset {\mathcal E}_{\le \alpha_{2}} \subset \dots \subset {\mathcal E}_{\le \alpha_{M}}={\mathcal E}_{i}
\]
of holomorphic sub-bundles of  $ {\mathcal E}_{i}$ and ${\mathcal E}_{\le \alpha_{j}}/{\mathcal E}_{\le \alpha_{j-1}}\cong {\mathcal E}_{\alpha_{j}}$.

\begin{prop}{\rm (cf. \cite[Section 4.8]{Ak})}\label{maro}
If $\{X_{1},\dots, X_{l}\}$ is  non-resonant  for every $\alpha_{j}\in \Delta_{i}\cap \Delta^{+}$,
then the quotient map $ {\mathcal E}_{i}\to {\mathcal E}_{\le \alpha_{M}}/{\mathcal E}_{\le \alpha_{M-1}}$
induces a cohomology isomorphism 
\[H^{\ast}(K,{\mathcal E}_{i})\cong H^{\ast}(K, {\mathcal E}_{\alpha_{M}})\cong H^{\ast}(\bar{\fl^{\epsilon}}, \langle  v_{- \alpha_{M}}\rangle\otimes V^{\alpha_{M}}(0)\otimes \langle [ E_{\alpha_{M}}]\rangle).
\]

\end{prop}
\begin{proof}
For each $k$, we have the long exact sequence
\[\xymatrix{0\ar[r]&H^{0}(K, {\mathcal E}_{\le \alpha_{k-1}}) \ar[r]& H^{0}(K, {\mathcal E}_{\le \alpha_{k}}) \ar[r] & H^{0}(K, {\mathcal E}_{ \alpha_{k}}) \ar[r]& H^{1}(K, {\mathcal E}_{\le \alpha_{k-1}}) \ar[r]&\dots .
}
\]

In case $R(\alpha_{k}, X_{1},\dots, X_{l})=\{(\sigma, \lambda, 0)\}$ if and only if $\alpha_{k}=\alpha_{M}$, 
by Theorem \ref{BWBD},
the proposition easily follows from  an iteration of computations on long exact sequences.

We consider the case $R(\alpha_{k}, X_{1},\dots, X_{l})=\{(\sigma, \lambda, 0)\}$ holds  for $\alpha_{k}\not=\alpha_{M}$.
For such $\alpha_{k}$,  $\alpha_{k}+ \frac{1}{2}\sum_{\Delta^{+}}\alpha$ is regular.
This happens only if the type of $\Delta_{i}$ is $B_{r}$, $C_{r}$, $F_{4}$ or $G_{2}$. (see \cite[Section 4.7]{Ak}).
Suppose $\Delta_{i}$ is $B_{r}$, $C_{r}$ or  $F_{4}$.
Then, as the table of \cite[Page 129]{Ak},  we have  $\alpha^{\prime}, \alpha^{\prime\prime}\in \Delta^{+}-\{\alpha_{M}\}$ such that 
\begin{enumerate}
\item $\alpha^{\prime}=\alpha^{\prime\prime}+\beta$ for some simple root $\beta$.
\item For $\alpha_{k}\in\Delta_{i}\cap \Delta^{+} $,  $\alpha_{k}+ \frac{1}{2}\sum_{\Delta^{+}}\alpha$ is regular if and only if $\alpha_{k}=\alpha^{\prime}$, $\alpha^{\prime\prime}$ or $\alpha_{M}$. $\alpha_{k}\in D$ if and only  if $\alpha_{k}=\alpha^{\prime}$ or $\alpha_{M}$.
\item $\sigma_{\beta}(\alpha^{\prime})=\alpha^{\prime}$ where $\sigma_{\beta}\in W$ is the reflection corresponding to $\beta$.
\end{enumerate}
By the first condition, we can take an ordering such that $\alpha_{j}=\alpha^{\prime}$ and $\alpha_{j-1}=\alpha^{\prime\prime}$ for some $j$.
By the second and third condition, by   an iteration of computations on long exact sequences with Theorem \ref{BWBD}, the quotient ${\mathcal E}_{\le \alpha_{j-1}}\to {\mathcal E}_{\le \alpha_{j-1}}/{\mathcal E}_{\le \alpha_{j-2}}$ induces a cohomology isomorphism
\[H^{\ast}(K, {\mathcal E}_{\le \alpha_{j-1}})\cong H^{\ast-1}(\bar{\fl^{\epsilon}}, \langle E_{\beta}\otimes v_{- \alpha^{\prime}}\rangle\otimes V^{\alpha^{\prime}}(0)\otimes \langle [ E_{\alpha^{\prime\prime}}]\rangle)
\]
and we have 
\[H^{\ast}(K, {\mathcal E}_{ \alpha_{j}})\cong  H^{\ast}(\bar{\fl^{\epsilon}}, \langle  v_{- \alpha^{\prime}}\rangle\otimes V^{\alpha^{\prime}}(0)\otimes \langle [ E_{\alpha^{\prime}}]\rangle).
\]
Thus, for the long exact sequence
\[\xymatrix{0\ar[r]&H^{0}(K, {\mathcal E}_{\le \alpha_{j-1}}) \ar[r]& H^{0}(K, {\mathcal E}_{\le \alpha_{j}}) \ar[r] & H^{0}(K, {\mathcal E}_{ \alpha_{j}}) \ar[r]& H^{1}(K, {\mathcal E}_{\le \alpha_{j-1}}) \ar[r]&\dots 
}, 
\]
we can easily say that 
the boundary  map $H^{\ast}(K, {\mathcal E}_{ \alpha_{j}})\to H^{\ast+1}(K, {\mathcal E}_{\le \alpha_{j-1}})$ is an isomorphism and hence $H^{\ast}(K, {\mathcal E}_{\le \alpha_{j}})=0$.
Now, by the first condition of $ \alpha^{\prime}, \alpha^{\prime\prime}$, the proposition follows from  an iteration of computations on long exact sequences.

Suppose $\Delta_{i}$ is $G_{2}$.
Then, for simple roots $\alpha$, $\beta$ corresponding to the Cartan matrix $\left(
\begin{array}{ccc}
2& -3  \\
-1&    2
\end{array}
\right)$, we have
\[\Delta_{i}\cap \Delta^{+}=\{\alpha_{1}=\alpha,\alpha_{2}=\beta ,\alpha_{3}=\alpha+\beta,\alpha_{4}=2\alpha+\beta,\alpha_{5}=3\alpha+\beta,\alpha_{6}=3\alpha+2\beta\}.
\]
By  the table of \cite[Page 129]{Ak}, $\alpha_{k}+ \frac{1}{2}\sum_{\Delta^{+}}\alpha$ is regular if and only if $\alpha_{k}=\alpha_{2}$, $\alpha_{4}$ or $\alpha_{6}$ and 
 $\alpha_{k}\in D$ if and only  if $\alpha_{k}=\alpha_{4}$ or $\alpha_{6}$.
Since $\sigma_{\alpha}(\alpha_{4})=\alpha+\beta$, by the long exact sequence,   the quotient ${\mathcal E}_{\le \alpha_{2}}\to {\mathcal E}_{\le \alpha_{2}}/{\mathcal E}_{\le \alpha_{1}}$ induces a cohomology isomorphism
\[H^{\ast}(K, {\mathcal E}_{\le \alpha_{2}})\cong H^{\ast-1}(\bar{\fl^{\epsilon}}, \langle E_{\alpha}\otimes v_{- \alpha-\beta}\rangle\otimes V^{\alpha_{4}}(0)\otimes \langle [ E_{\beta}]\rangle)
\]
and   the quotient ${\mathcal E}_{\le \alpha_{4}}/{\mathcal E}_{\le \alpha_{2}}\to {\mathcal E}_{\le \alpha_{4}}/{\mathcal E}_{\le \alpha_{3}}$
 induces a cohomology isomorphism
\[H^{\ast}(K, {\mathcal E}_{\le \alpha_{4}}/{\mathcal E}_{\le \alpha_{2}})\cong H^{\ast}(\bar{\fl^{\epsilon}}, \langle v_{- 2\alpha-\beta}\rangle\otimes V^{\alpha_{4}}(0)\otimes \langle [ E_{2\alpha+\beta}]\rangle).
\]
We note that  $\langle E_\alpha\otimes v_{-\alpha-\beta}\otimes [E_\beta]\rangle$ and $\langle v_{-2\alpha-\beta}\otimes [E_{2\alpha+\beta}]\rangle $ are trivial $\bar{\fl^{\epsilon}}$-modules.
Thus,  for the long exact sequence 
\[\xymatrix{0\ar[r]&H^{0}(K, {\mathcal E}_{\le \alpha_{2}}) \ar[r]& H^{0}(K, {\mathcal E}_{\le \alpha_{4}}) \ar[r] & H^{0}(K, {\mathcal E}_{\le \alpha_{4}}/{\mathcal E}_{\le \alpha_{2}}) \ar[r]& H^{1}(K, {\mathcal E}_{\le \alpha_{2}}) \ar[r]&\dots 
}, 
\]
we can easily say that  the boundary  map $H^{\ast}(K,{\mathcal E}_{\le \alpha_{4}}/{\mathcal E}_{\le \alpha_{2}})\to H^{\ast+1}(K, {\mathcal E}_{\le \alpha_{2}})$ is an isomorphism and hence $H^{\ast}(K, {\mathcal E}_{\le \alpha_{4}})=0$.
Hence the proposition follows from the  long exact sequences for ${\mathcal E}_{\le \alpha_{5}}$ and ${\mathcal E}_{\le \alpha_{6}}$.
\end{proof}

We identify $\g\cong \g_{hol}^{R}$.
Consider the decomposition $\g=\g_{1}\oplus\dots \oplus \g_{m}$ associated with $\Delta=\Delta_{1}\sqcup\dots, \sqcup\Delta_{m}$.
Then,   each $\g_{i}$ is an irreducible $\g$-module  $V^{\alpha_{M}}$ corresponding to $\alpha_{M}\in D$.
This proposition with Remark \ref{regen} implies the following theorem.
\begin{thm}\label{HT}
If $\{X_{1},\dots, X_{l}\}$ is  non-resonant  for every $\alpha\in \Delta^{+}$ and the trivial representation $0$, then
we have an isomorphism of $\C$-vector space
\[H^{\ast}(K,T^{1,0}_{J_{\fl+\epsilon}}K)\cong \bigwedge  \C^{l}\otimes   \C^{l}\oplus H^{\ast}(\C^{l}, \g(0))
\]
where $l=\frac{1}{2}{\rm rank }\,K$, $\g(0)$ is the  generalized eigenspace corresponding to the eigenvalue $0$ for the adjoint operators of $X_{1},\dots, X_{l}$ and  $H^{\ast}(\C^{l}, \g(0))$ is the Lie algebra cohomology of the abelian Lie algebra  $\C^{l}$ with values in the representation $\C^{l}\to {\rm End}(\g(0))$ so that for the standard basis $e_{1},\dots e_{l}$ of  $\C^{l}$, each $e_{i}$ represents the adjoint operator of $X_{i}$. 
In particular
\begin{itemize}
\item We have an isomorphism of $\C$-vector space
 \[H^{0}(K,T^{1,0}_{J_{\fl+\epsilon}}K)\cong     \C^{l}\oplus  \g^{\langle X_{1},\dots, X_{l}\rangle}.\]
\item If $X_{1},\dots, X_{l}$ are contained in a same Cartan subalgebra of $\g$, then 
\[H^{\ast}(K,T^{1,0}_{J_{\fl+\epsilon}}K)\cong \bigwedge  \C^{l}\otimes   (\C^{l}\oplus  \g^{\langle X_{1},\dots, X_{l}\rangle}).
\]
\end{itemize}
\end{thm}

\begin{cor}\label{nobil}
If $\{X_{1},\dots, X_{l}\}\not=\{0,\dots, 0\}$ and $\{X_{1},\dots, X_{l}\}$ is  non-resonant  for every $\alpha\in \Delta^{+}$ and the trivial representation $0$, then the complex manifold $(K, J_{\fl+\epsilon})$ is not biholomorphic to $K$ with  any left-invariant or right-invariant complex structure.

\end{cor}
\begin{proof}
By Proposition \ref{cohlI}, the dimension of the vector space of holomorphic vector fields on $K$ with a left-invariant (resp. right-invariant) complex structure is $l +\dim \g$ and we have
\[l +\dim \g>l+\dim  \g^{\langle X_{1},\dots, X_{l}\rangle}= \dim H^{0}(K,T^{1,0}_{J_{\fl+\epsilon}}K)).\]

\end{proof}

As we noticed in Remark \ref{regen}, for each $\alpha\in \Delta^{+}$ the non-resonant condition is generic.
Hence we can say that there exists a small deformation of a left-invariant complex structure on $K$ which is not biholomorphic to  any left-invariant or right-invariant complex structure.

\begin{ex}
We consider Example \ref{CE}.
Take $A_{1}=-bT_{2}+i(T_{1}+aT_{2})$ and small $X_{1}\in \g^{R}_{hol}$.
We have $\Delta^{+}=\{2it_{1}\}\cup \{ 2it_{2}\}$ and $W=\{\pm 1\}\times \{\pm 1\}$.
Any irreducible representation $V^{\lambda}$ of $K$ corresponds to $\lambda= mit_{1}+nit_{2}$ for $m,n\in \N$.
We can easily say that if $X_{1}$ is sufficiently small, then $X_{1}$ is  non-resonant  for every $\alpha\in \Delta^{+}$ and the trivial representation $0$.
By the above result, if  $X_{1}$ is sufficiently small, then we have an isomorphism of $\C$-vector space
\[H^{0}(K,T^{1,0}_{J_{\fl+\epsilon}}K)\cong \C \oplus ({\frak sl}_{2}(\C))^{\langle X_{1}\rangle} \oplus ({\frak sl}_{2}(\C))^{\langle X_{1}\rangle}.
 \]

In the resonant case, Theorem \ref{HT}  does not hold.
Let $b=1$, $a=0$ and  $X_{1}=\frac{1}{3}T^{R}_{1}$.
Then we have
\[v^{3it^{R}_{1}}\otimes \left(v^{-3it_{1}}\otimes E_{2it_{1}}+ E_{-2it_{1}}\cdot (v^{-3it_{1}})\otimes p^{1,0}_{\epsilon}([E_{2it_{1}}, E_{-2it_{1}}]) \right)\in H^{0}(K,T^{1,0}_{J_{\fl+\epsilon}}K)
\]
where $v^{\pm 3it_{1}}$ and $v^{\pm 3it^{R}_{1}}$ are highest (lowest) weight vectors in representations  $V^{3it_{1}}$ and $V^{3it^{R}_{1}}$ with highest  (lowest) weights $\pm 3it_{1}$ and $\pm 3it^{R}_{1}$ respectively.
Consider $G/U^{-}= (\C^{2}-\{0\})\times (\C^{2}-\{0\})$.
Then  the complex manifold $(K, J_{\fl+\epsilon})$ is  given by the transversality between  $S^{3}\times S^{3}$ and  $\left\langle \frac{4}{3}z_{1}\frac{\partial}{\partial z_{1}}+\frac{2}{3}z_{2}\frac{\partial}{\partial z_{2}}+w_{1}\frac{\partial}{\partial w_{1}}+w_{2}\frac{\partial}{\partial w_{2}}\right\rangle$.
We have a non-trivial holomorphic vector field on $S^{3}\times S^{3}$ with such complex structure which is given by the $(\frac{4}{3},\frac{2}{3},1,1)$-resonant vector field $z_{2}^{2}\frac{\partial}{\partial z_{1}}$ (see \cite{Ha},\cite{LN}).

\end{ex}

\subsection{The Lie algebra of holomorphic vector fields and the group of automorphisms}
Take the complex Lie subgroup $H\subset G$ corresponding to $\h$.
Consider the action of $G\times H$ on the complex homogeneous space $G/U^{-}$ so that for $(g,h)\in G\times H$ and $[x]\in G/U^{-}$, the action is given by $(g,h)\cdot [x]=[g^{-1}xh]$.
Take the subgroup 
\[\exp\left(\bar{\fl^{\epsilon}}\right)=\left\{\left(\exp\left(\sum_{t=1}^{l} t_{i}X_{i}\right), \exp\left(\sum_{t=1}^{l} t_{i}\bar{A}_{i}\right)\right)\in G\times H\relmiddle| (t_{1},\dots ,t_{l})\in \C^{l}\right\}
\]
in $G\times H$.
We assume that  every $\exp\left(\bar{\fl^{\epsilon}}\right)$-orbit intersects with $K$ at only one point in $G/U^{-}$.
Then, the quotient $(G/U^{-})/ \exp\left(\bar{\fl^{\epsilon}}\right)$ is biholomorphic to the complex manifold $(K, J_{\fl+\epsilon})$.
Denote by $G^{\langle X_{1},\dots ,X_{l}\rangle}$ the subgroup in $G$ consisting of $g\in G$ such that the adjoint operator of $g$ fixes $ X_{1},\dots ,X_{l}$.
The Lie algebra of this Lie group is $\g^{\langle X_{1},\dots ,X_{l}\rangle}$.
Since $G^{\langle X_{1},\dots ,X_{l}\rangle}\times H$ commutes  with $\exp\left(\bar{\fl^{\epsilon}}\right)$ in $G\times H$,
we have the holomorphic action of  $G^{\langle X_{1},\dots ,X_{l}\rangle}\times H$ on  $(G/U^{-})/ \exp\left(\bar{\fl^{\epsilon}}\right)\cong (K, J_{\fl+\epsilon})$.
Denote by ${\rm Aut}(K, J_{\fl+\epsilon})$ the group of automorphisms of the complex manifold $ (K, J_{\fl+\epsilon})$.
We have the homomorphism $\psi^{\epsilon}:G^{\langle X_{1},\dots ,X_{l}\rangle}\times H\to {\rm Aut}(K, J_{\fl+\epsilon})$.
Since $G$ is semisimple,  we can easily compute the kernel  of this homomorphism as
\[{\rm ker}(\psi^{\epsilon})=\{(az,  bz^{-1}) \mid (a,b)\in \exp\left(\bar{\fl^{\epsilon}}\right), z\in Z(G) \}
\]
and hence its Lie algebra is $\bar{\fl^{\epsilon}}=\langle \bar{A}_{1}+X_{1},\dots \bar{A}_{l}+X_{l}\rangle \subset \g^{\langle X_{1},\dots ,X_{l}\rangle}\oplus \h$.
Thus, we have the embedding of $(\g^{\langle X_{1},\dots ,X_{l}\rangle}\oplus \h)/\bar{\fl^{\epsilon}}$ into $H^{0}(K,T^{1,0}_{J_{\fl+\epsilon}}K)$ as a Lie subalgebra.
By Theorem \ref{HT}, we have:
\begin{thm}
If $\{X_{1},\dots, X_{l}\}\not=\{0,\dots, 0\}$ and $\{X_{1},\dots, X_{l}\}$ is  non-resonant  for every $\alpha\in \Delta^{+}$ and the trivial representation $0$ and every $\exp\left(\bar{\fl^{\epsilon}}\right)$-orbit intersects with $K$ at only one point in $G/U^{-}$, 
then  we have a Lie algebra isomorphism
\[H^{0}(K,T^{1,0}_{J_{\fl+\epsilon}}K)\cong (\g^{\langle X_{1},\dots ,X_{l}\rangle}\oplus \h)/\bar{\fl^{\epsilon}}.
\]
\end{thm}

\begin{cor}
If $\{X_{1},\dots, X_{l}\}\not=\{0,\dots, 0\}$ and $\{X_{1},\dots, X_{l}\}$ is  non-resonant  for every $\alpha\in \Delta^{+}$ and the trivial representation $0$ and every $\exp\left(\bar{\fl^{\epsilon}}\right)$-orbit intersects with $K$ at only one point in $G/U^{-}$, then
the image  ${\rm im} (\psi^{\epsilon})$ of the   homomorphism $\psi^{\epsilon}:G^{\langle X_{1},\dots ,X_{l}\rangle}\times H\to {\rm Aut}(K, J_{\fl+\epsilon})$ contains the identity component of  the group ${\rm Aut}(K, J_{\fl+\epsilon})$ of automorphisms of the complex manifold $ (K, J_{\fl+\epsilon})$.
\end{cor}

\begin{ex}
We consider Example \ref{CE}.
Take $A_{1}=-bT_{2}+i(T_{1}+aT_{2})$  and small $X_{1}\in \g^{R}_{hol}$.
As a global distribution on $G/U^{-}\cong (\C^{2}-\{0\})\times (\C^{2}-\{0\})$, we have
\[\bar{\fl^{\epsilon}}=\left\langle z_{1}\frac{\partial}{\partial z_{1}}+z_{2}\frac{\partial}{\partial z_{2}}+c(w_{1}\frac{\partial}{\partial w_{1}}+w_{2}\frac{\partial}{\partial w_{2}})+\sum_{i,j\le 2}\left(a_{ij} z_{i}\frac{\partial}{\partial z_{j}}+b_{ij} w_{i}\frac{\partial}{\partial w_{j}}\right)\right\rangle\] 
with $c=a+bi$ for complex numbers  $a_{ij}, b_{ij}$ with $a_{11}+a_{22}=0$, $b_{11}+b_{22}=0$.
By \cite[Corollary 2]{LN}, for sufficiently small $X_{1}$, 
$\exp\left(\bar{\fl^{\epsilon}}\right)$-orbit intersects with $K=S^{3}\times S^{3}$ at only one point.
Thus we actually have a Lie algebra isomorphism
\[H^{0}(K,T^{1,0}_{J_{\fl+\epsilon}}K)\cong \C \oplus ({\frak sl}_{2}(\C))^{\langle X_{1}\rangle} \oplus ({\frak sl}_{2}(\C))^{\langle X_{1}\rangle}.
 \]

\end{ex}


\begin{thebibliography}{40}
\bibitem{Aka} K.  Akao, 
On deformations of the Calabi-Eckmann manifolds. Proc. Japan Acad. 51 (1975), no. 6, 365--368.
\bibitem{Ak}
D. N. Akhiezer,
Lie group actions in complex analysis. 
Aspects of Mathematics, {\bf E27}. Friedr. Vieweg  Sohn, Braunschweig, 1995.
\bibitem{AK} D. Angella, H. Kasuya
Cohomologies of deformations of solvmanifolds and closedness of some properties. North-West. Eur. J. Math. {\bf 3} (2017), 75--105.
\bibitem{B}
R. Bott, Homogeneous vector bundles, Ann. of Math. (2) {\bf 66} (1957), 203--248.
\bibitem{CE}
E. Calabi, B. Eckmann,  A class of compact complex manifolds which are not algebraic. Ann, of Math., {\bf 58}, 494--500 (1953).
\bibitem{FH}
W. Fulton, J. Harris, Representation Theory, GTM {\bf 129}, Springer-Verlag, 1991.
\bibitem{Ghy}
E. Ghys,  D\'eformations des structures complexes sur les espaces homog\'enes de SL(2, C). J. Reine Angew. Math. {\bf468} (1995), 113--138.

\bibitem{Ha}
A. Haefliger, 
Deformations of transversely holomorphic flows on spheres and deformations of Hopf manifolds. 
Compositio Math. {\bf 55} (1985), no. 2, 241--251. 
\bibitem{Has} K. 
Hasegawa, Small deformations and non-left-invariant complex structures on six-dimensional compact solvmanifolds. Differ. Geom. Appl. {\bf 28}(2), 220--227 (2010)
\bibitem{He}
S. Helgason,  Differential geometry, Lie groups, and symmetric spaces. 
 Graduate Studies in Mathematics, {\bf 34}. American Mathematical Society, Providence, RI, 2001.
 
 \bibitem{Hir}
F. Hirzebruch,  Topological methods in algebraic geometry.    Classics in Mathematics. Springer-Verlag, Berlin, 1995.
 
 \bibitem{Huy}
 D. Huybrechts, 
Complex geometry. 
An introduction. Universitext. Springer-Verlag, Berlin, 2005. 

\bibitem{Kasd} H. Kasuya, 
de Rham and Dolbeault cohomology of solvmanifolds with local systems. Math. Res. Lett. {\bf 21} (2014), no. 4, 781--805.

\bibitem{Kos} B. Kostant, Lie algebra cohomology and the generalized Borel-Weil theorem. Ann. of Math. (2) {\bf 74} (1961) 329--387. 

\bibitem{LN} J.J. Loeb, M. Nicolau 
Holomorphic flows and complex structures on products of odd-dimensional spheres. 
Math. Ann. {\bf306} (1996), no. 4, 781--817. 

\bibitem{LMN}  J. J. Loeb, M. Manjarin, M. Nicolau,
Complex and CR structures on compact Lie groups associated to abelian actions. 
Ann. Global Anal. Geom. {\bf 32} (2007), no. 4, 361--378.
\bibitem{Nak}I. 
Nakamura,  Complex parallelisable manifolds and their small deformations. J. Differ. Geom. {\bf10}(1), 85--112 (1975)
\bibitem{Rol}
S.  Rollenske,  The Kuranishi space of complex parallelisable nilmanifolds. J. Eur. Math. Soc. (JEMS) {\bf13} (2011), no. 3, 513--531.
\bibitem{Sam}  H. Samelson, A class of compact-analytic manifolds. Portugaliae  Math. {\bf12}  (1953) 129--132.
\bibitem{Wangp}H. 
Wang,  Complex parallelisable manifolds. Proc. Amer. Math. Soc. {\bf 5} (1954), 771--776.
\bibitem{Wang} H. Wang, Closed manifolds with homogeneous complex structure. Amer. J. of Math. {\bf 76} (1954) 1--32
\bibitem{Win}
J. Winkelmann,  Complex-analytic geometry of complex parallelizable manifolds. Habilitationsschrift. Heft 13 der Schriftenreihe des Graduiertenkollegs Geometrie und Mathematische Physik an der Ruhr-Universit\"at Bochum (1995).
 \end{thebibliography}
\end{document}